\documentclass[paper=a4, headings=small]{scrartcl}

\usepackage[T1]{fontenc}
\usepackage[utf8]{inputenc}
\usepackage[english]{babel}
\usepackage{graphicx}

\usepackage{amsmath}
\usepackage{amssymb}
\usepackage{amstext}
\usepackage[amsmath,thmmarks]{ntheorem}

\usepackage{tikz}

\usepackage{hyperref}

\theoremseparator{.}
\newtheorem{thm}{Theorem}
\newtheorem{lem}[thm]{Lemma}

\newtheorem{prop}[thm]{Proposition}

\theorembodyfont{\upshape}

\newtheorem*{ex}{Example}

\theoremstyle{nonumberplain}
\theoremheaderfont{\itshape}
\theoremseparator{.}
\theoremsymbol{\ensuremath{_\square}}
\newtheorem{proof}{Proof}

\renewcommand{\P}{$\mathcal{P}$ }
\newcommand{\Q}{$\mathcal{Q}$ }
\newcommand{\mathP}{\mathcal{P}}
\newcommand{\mathQ}{\mathcal{Q}}
\newcommand{\minP}[1]{$0_\mathcal{#1}$}
\newcommand{\maxP}[1]{$1_\mathcal{#1}$}

\title{A recursive description of automorphism groups of inductively constructed polytopes}
\author{Lara Be{\ss}mann \thanks{Funded as project KR1668/10 by the Deutsche Forschungsgemeinschaft, and under Germany's Excellence Strategy EXC 2044-390685587, Mathematics M\"unster: Dynamics-Geometry-Structure.}}

\begin{document}
\maketitle 
\begin{abstract}
Polytopes are ubiquitous in different areas of mathematics. 
Gleason and Hubard established a factorisation theorem, stating that every abstract polytope has a unique factorisation into prime polytopes. 
We compute the automorphism group of a certain family of inductively constructed polytopes using the unique factorisation.  
\end{abstract}

In various fields of mathematics, polytopes like prisms and pyramids occur. 
For example in numerical mathematics they are used for tessellations of spaces.
For this, polytopes in certain dimensions are needed and there are different ways to construct them. 
We study an inductively generated family of polytopes, consisting of prisms and pyramids such that in the $n$-th step $2^n$ polytopes are constructed. 
One question is then if the automorphism groups of these inductively constructed polytopes can be computed. 

An abstract polytope is a partially ordered set (poset). 
Hence, to understand the structure and possibly find a factorisation of abstract polytopes it is necessary to understand first posets.
Moreover, to study the automorphism group of an abstract polytope it is helpful to understand the automorphisms of the posets they are based on. 
Posets and their automorphism groups have been studied, among others, by Hashimoto in \cite{Hashimoto} and Duffus in \cite{Duffus}.
Their work is extended by Gleason and Hubard to abstract polytopes.

Hashimoto's Refinement theorem \cite[Theorem 1]{Hashimoto} states that every connected poset has a unique factorisation into irreducible posets.  
Gleason and Hubard used this in \cite{GleasonHubard} to obtain unique prime factorisations for every abstract polytope.  
We concentrate on the Cartesian product $\times$ and the join $*$.
The work of Gleason and Hubard leads to two unique prime factorisations for every abstract polytope, one with respect to the Cartesian product and one with respect to the join. 

Furthermore, Duffus studied the automorphism groups of posets and showed that every automorphism of a product of relatively prime posets is the product of automorphisms of each factor.   
Gleason and Hubard connected the prime factorisation using the result by Duffus to describe the automorphism groups of a product of polytopes as the product of the automorphism groups of the factors. 

Using the prime factorisation and the splitting of the automorphism groups of abstract polytopes obtained by Gleason and Hubard we compute the automorphism groups of the following inductively constructed family of polytopes. (See Figure \ref{GenProcess} for an illustration.)
\begin{quote}
Starting with an edge, in each step, from every polytope two new polytopes are constructed by taking the join with a point and the Cartesian product with an edge.
\end{quote}
\begin{figure}[htb] \centering 
\begin{tikzpicture}[point/.style={circle,draw,fill=black,inner sep=0.5pt},scale=0.7]

\draw (-6,-0.5) node[point](v1){} -- (-6,0.5) node[point](v2){};

\draw (0,3.5) node[point](v3){} -- (0,4.5) node[point](v4){} -- (1,4.5) node[point](v5){} -- (1,3.5) node[point](v6){} -- cycle;

\draw (0,-3.5) node[point](v7){} -- (0,-4.5) node[point](v8){} -- (0.75,-4) node[point](v9){} -- cycle;

\draw (6,5) node[point](v10){} -- (6,6) node[point](v11){} -- (7,6) node[point] (v12){} -- (7,5) node[point](v13){} -- (6,5) -- (6.5,5.5) node[point] (v14){} -- (6.5,6.5) node[point](v15){} -- (7.5,6.5) node[point] (v16) {} -- (7.5,5.5) node[point](v17) {} -- (6.5,5.5) -- cycle;
\draw (6,6) -- (6.5,6.5);
\draw (7,6) -- (7.5,6.5);
\draw (7,5) -- (7.5,5.5);

\draw (6,1.5) node[point](v18){} -- (6.5,2) node[point](v19){} -- (7.5,2) node[point](v20){} -- (7,1.5) node[point](v21){} -- cycle;
\draw (6,1.5) -- (6.7,3) node[point](v22){};
\draw (7,1.5) -- (6.7,3);
\draw (6.5,2) -- (6.7,3);
\draw (7.5,2) -- (6.7,3);

\draw (6,-3) node[point](v23){} -- (7,-3) node[point](v24){} -- (6.5,-2.25) node[point](v25){} -- cycle;
\draw (6.5,-2.5) node[point](v26){} -- (7.5,-2.5) node[point](v27){} -- (7,-1.75) node[point](v28){} -- cycle;
\draw (6,-3) -- (6.5,-2.5);
\draw (7,-3) -- (7.5,-2.5);
\draw (6.5,-2.25) -- (7,-1.75);

\draw (6.2,-5.5) node[point](v29){} -- (6.2,-6.5) node[point](v30){} -- (7.2,-6) node[point](v31){} -- cycle;
\draw (6.2,-5.5) -- (7.2, -5);
\draw (6.2,-6.5) -- (7.2, -5);
\draw (7.2,-6) -- (7.2, -5);

\draw[->] (-5.5,0.5) -- (-0.5,3.6); 
\draw[->] (-5.5,-0.5) -- (-0.5,-3.6); 
\draw[->] (1.5,4.3) -- (5.5,5.4); 
\draw[->] (1.5,3.7) -- (5.5,2.6); 
\draw[->] (1.5,-3.7) -- (5.5,-2.6); 
\draw[->] (1.5,-4.3) -- (5.5,-5.4); 
\draw[->] (8,6) -- (10,6.5); 
\draw[->] (8,5.5) -- (10,5); 
\draw[->] (8,2.5) -- (10,3); 
\draw[->] (8,2) -- (10,1.5); 
\draw[->] (8,-2.125) -- (10,-1.625); 
\draw[->] (8,-2.625) -- (10,-3.125); 
\draw[->] (8,-5.5) -- (10,-5); 
\draw[->] (8,-6) -- (10,-6.5); 

\draw (-3.5,2.25) node{\small $\times$} [align=center];
\draw (-3.5,-2.25) node{\small $*$} [align=center];
\draw (3.5,5.25) node{\small $\times$} [align=center];
\draw (3.5,2.75) node{\small $*$} [align=center];
\draw (3.5,-2.75) node{\small $\times$} [align=center];
\draw (3.5,-5.25) node{\small $*$} [align=center];
\draw (9,6.5) node{\tiny $\times$} [align=center];
\draw (9,5) node{\tiny $*$} [align=center];
\draw (9,3) node{\tiny $\times$} [align=center];
\draw (9,1.5) node{\tiny $*$} [align=center];
\draw (9,-1.625) node{\tiny $\times$} [align=center];
\draw (9,-3.125) node{\tiny $*$} [align=center];
\draw (9,-5) node{\tiny $\times$} [align=center];
\draw (9,-6.5) node{\tiny $*$} [align=center];

\draw (11,6.5) node{$ \dots $} [align=center];
\draw (11,5) node{$ \dots $} [align=center];
\draw (11,3) node{$ \dots $} [align=center];
\draw (11,1.5) node{$ \dots $} [align=center];
\draw (11,-1.625) node{$ \dots $} [align=center];
\draw (11,-3.125) node[align=center]{$\dots$};
\draw (11,-5) node{$ \dots $} [align=center];
\draw (11,-6.5) node{$ \dots $} [align=center];
\end{tikzpicture}
\caption{The construction process of the regarded polytopes.} \label{GenProcess}
\end{figure}
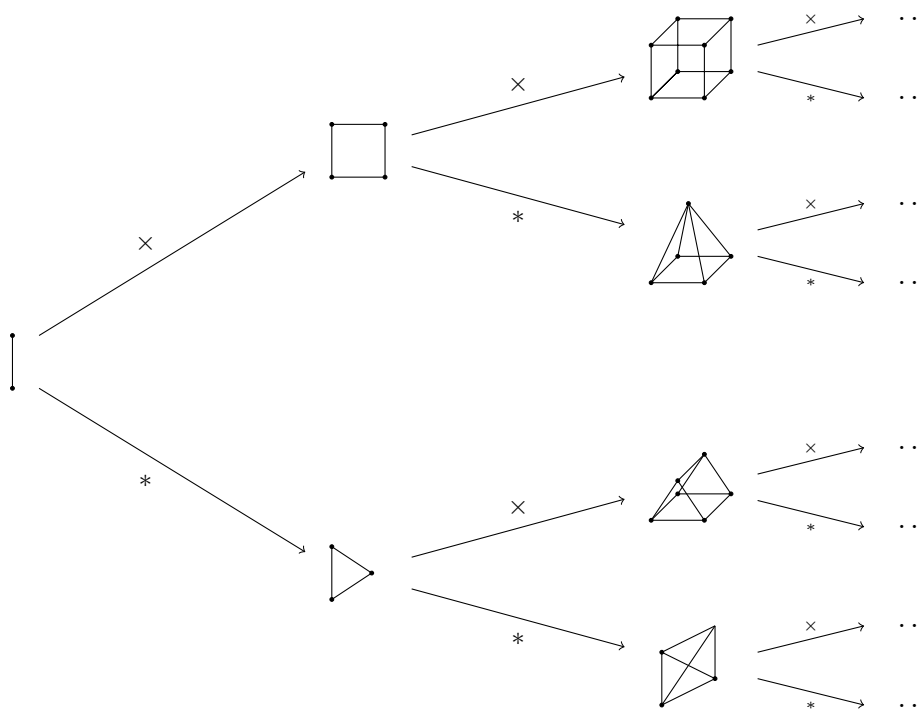

First, we are going to recall the relevant definitions and results. 
Then we prove that a prism is not a pyramid and a pyramid is not a prism. 
This is crucial to describe later the automorphism groups in terms of the previous polytope. 
Finally, we compute the automorphism groups. 
These automorphism groups can be written in terms of direct and semi-direct products of symmetric groups.
At the end, we give an algorithm attaching to each polytope the data that is necessary to compute its automorphism group.

\section{Abstract polytopes}
For the definition of abstract polytopes we follow Section 2A in \cite{McMullenSchulte}.
The notion of abstract polytopes generalise convex polytopes and every convex polytope in $\mathbb{R}^d$ is an abstract polytope (cf. \cite[Section 2A]{McMullenSchulte}). 

An \textit{abstract polytope} \P of rank $n$ is a partially ordered set (poset) with the following properties.
\begin{itemize}
\item It has a minimal element \minP{P} and a maximal element \maxP{P}. 
\item Every maximal totally ordered subset contains exactly $n+2$ elements including \minP{P} and \maxP{P}. This leads to a rank function $rank_{\mathP} \colon \mathP \to \{{-1},0,1, \dots, n \}$, the minimal element \minP{P} has rank ${-1}$ and the maximal element \maxP{P} has rank $n$.
An element of rank 0 is called a \textit{vertex} and an element of rank 1 is called an \textit{edge}.
\item A \textit{section} for two elements $F \leq G$ in \P is defined as
$ G/F = \left\{ H \mid F \leq H \leq G \right\}$.  

If either $rank(G)-rank(F) \leq 2$ or if for any two elements $F',G' \in G/F$ different from $F$ and $G$, there is a sequence of elements
\begin{align*}
F' = H^0, H^1, H^2, \dots, H^k=G'
\end{align*}
such that $H^i \in G/F$ for all $i=0, \dots, k$ and $H^i \leq H^{i+1}$ holds for every $i=0, \dots, k-1$, then the section $G/F$ is called \textit{connected}.

Every section in $\mathP$, including $\mathP = 1_{\mathP}/0_{\mathP}$ itself, is connected.
\item For incident elements $F \leq G$ with $rank(G)-rank(F)=2$ there are precisely two elements $H_1 \neq H_2$ with $F < H_1,H_2 < G$. This is the \textit{diamond condition}.
\end{itemize}

The polytopes $pt$ of rank $0$ and $I$ of rank $1$ are unique up to isomorphism.

\subsubsection*{Products of abstract polytopes}
The polytopes we are interested in are constructed using joins and Cartesian products.
Gleason and Hubard prove in section 4 in \cite{GleasonHubard} that joins and Cartesian products of abstract polytopes are again abstract polytopes.
We recall the definition of the join $*$ and the Cartesian product $\times$. 
To avoid repetition, let $\square \in \{*, \times\}$ either represent the join or the Cartesian product.

Let \P and \Q be abstract polytopes. The \textit{join} $\mathP * \mathQ$ and the \textit{Cartesian product} $\mathP \times \mathQ$  are defined as the following sets
\begin{align*}
\mathP * \mathQ &= \left\{ (F,G) \mid F \in \mathP, G \in \mathQ \right\} \\
\mathP \times \mathQ &= \left\{ (F,G) \mid F \in \mathP, G \in \mathQ, rank_{\mathP}(F), rank_{\mathQ}(G) \geq 0  \right\} \cup \left\{ (0_{\mathP}, 0_{\mathQ}) \right\}
\end{align*}
together with the following ordering
\begin{align*}
(F,G) \leq_{\mathP \square \mathQ} (F',G') \text{ if and only if } F \leq_{\mathP} F' \text{ and } G \leq_{\mathQ} G'.
\end{align*}
The rank of $(F,G) \in \mathP * \mathQ$ is given by
\begin{align*}
rank_{\mathP * \mathQ}(F,G) = rank_{\mathP}(F) + rank_{\mathQ}(G)+1. 
\end{align*}
The join of two abstract polytopes is an abstract polytope by \cite[Prop. 4.2]{GleasonHubard} and its rank is the sum of the ranks plus 1.

The rank of $(F,G)$ in $\mathP \times \mathQ$ with $rank_{\mathP}(F),rank_{\mathQ}(G) \geq 0$ is 
\begin{align*}
rank_{\mathP \times \mathQ}(F,G)= rank_{\mathP}(F) + rank_{\mathQ}(G)
\end{align*}
and the rank of $(0_{\mathP},0_{\mathQ})$ is -1. 
The Cartesian product of two abstract polytopes is an abstract polytope by \cite[Prop. 4.3]{GleasonHubard}. 
The rank of the product is the sum of the ranks. 

The polytope $\mathP^{\square k} = \mathP \square \dots \square \mathP$ for a natural number $k$, is the $k$-fold product of \P with itself. 
However, the polytope $pt^{* k}$ is a $k$-fold join of $pt$ and it is called the \textit{$k$-simplex}.
The polytope $I^{\times k}$ is a $k$-fold Cartesian product and called the \textit{$k$-cube}. 

Moreover, a \textit{pyramid} is a polytope of the form $\mathP * pt$ and a \textit{prism} is a polytope of the form $\mathP \times I$. 

An abstract polytope \P is \textit{prime} with respect to $\square$ if there exist no polytopes $\mathP_1$ and $\mathP_2$ such that $\mathP = \mathP_1 \square \mathP_2$. 

According to Hashimoto's Refinement theorem \cite[Theorem 1]{Hashimoto} every connected poset has a unique, up to isomorphism, factorisation into irreducible posets. 
Gleason and Hubard extend this result to abstract polytopes showing that each factor is an abstract polytope itself. 

\begin{thm}[Corollary 5.6 in \cite{GleasonHubard}]
An abstract polytope can be uniquely factorised as a product of abstract polytopes that are prime with respect to the decomposing product. 
\end{thm}
Hence every abstract polytope can be written as a unique product of polytopes, where each factor is prime with respect to $\square$. 
In our setting this means that every abstract polytope has two unique factorisations, one into a Cartesian product and one into a join. 

Two polytopes are \textit{relatively prime} with respect to $\square$ if the unique prime factorisations with respect to $\square$ of the polytopes have no prime polytopes in common.

An \textit{automorphism} of a polytope is an order-preserving permutation of its elements.
Duffus proved in Corollary 2 in \cite{Duffus} that an automorphism of a product of relatively prime posets is a product of automorphisms of each factor. 
Gleason and Hubard use this to obtain the following results. 

\begin{prop}\label{AutomProd}
\begin{enumerate}
\item Let \P and \Q be two polytopes that are relatively prime with respect to $\square$. 
Then 
\begin{align*}
Aut(\mathP \square \mathQ) \cong Aut(\mathP) \times Aut(\mathQ).
\end{align*} \label{AutProdRelPrimePoly}
\item Let \Q be a prime polytope with respect to $\square$ and let $\mathP = \mathQ \square \dots \square \mathQ = \mathQ^{\square m}$ be a product. 
Then 
\begin{align*}
Aut(\mathP) \cong  \left( \prod_{i=1}^m Aut(Q) \right) \rtimes Sym(m).
\end{align*} \label{AutProdPrimePoly}
\item Let $\mathP = \mathQ_1^{\square m_1} \square \dots \square \mathQ_r^{\square m_r}$ be an abstract polytope, where the $\mathQ_i$ are pairwise distinct prime polytopes with respect to $\square$.
Then 
\begin{align*}
Aut(\mathP) \cong  \prod_{i=1}^r \left( Aut(\mathQ_i)^{m_i} \rtimes Sym(m_i) \right).
\end{align*}\label{AutProdProdPrimePoly}
\end{enumerate}
\end{prop}

\begin{proof}
These are Propositions 7.1, 7.2, and 7.3 in \cite{GleasonHubard}.
\end{proof}

\begin{ex} 
We compute the automorphism groups of the polytopes of rank $0$ and $1$ using Proposition \ref{AutomProd}.

The polytope $pt$ is prime with respect to the join. 
Hence, by Proposition \ref{AutomProd}\ref{AutProdPrimePoly} it follows that
\begin{align*}
Aut(pt^{* k}) \cong \left( \prod_{i=1}^k Aut(pt) \right) \rtimes Sym(k) \cong Sym(k).
\end{align*}

The polytope $I$ is prime with respect to the Cartesian product and its automorphism group is isomorphic to $\mathbb{Z}/2\mathbb{Z}$.
For the automorphism group of the $k$-cube $I^{\times k}$ it follows again by Proposition \ref{AutomProd}\ref{AutProdPrimePoly} that
\begin{align*}
Aut(I^{\times k}) \cong \left( \prod_{i=1}^k Aut(I) \right) \rtimes Sym(k) \cong (\mathbb{Z}/2\mathbb{Z})^k \rtimes Sym(k). 
\end{align*}
\end{ex}
 
Using Proposition \ref{AutomProd} we observe that the automorphism group of a join of a polytope with the $k$-simplex splits as a direct product, if the polytope is relatively prime to $pt$. 
Similar, the same holds for the Cartesian product of a polytope with the $k$-cube, if the polytope is relatively prime to $I$. 

\begin{prop}  \label{AutNoPyrAndAutNoPrism}
Let \P be an abstract polytope.
If \P is not a pyramid, then 
\begin{align*}
Aut(\mathP * pt^{* k}) \cong Aut(\mathP)  \times Sym(k)
\end{align*}
for every $k \geq 1$, 
and if \P is not a prism then 
\begin{align*}
Aut(\mathP \times I^{\times k}) \cong Aut(\mathP) \times \left( (\mathbb{Z}/2\mathbb{Z})^k \rtimes Sym(k) \right)
\end{align*}
for every $k \geq 1$. 
\end{prop}

\begin{proof}
If \P is not a pyramid then it cannot be written as a join of an abstract polytope with $pt$. 
Hence, in the unique prime factorisation of \P with respect to the join the polytope $pt$ does not occur and thus \P and $pt$ are relatively prime. 
By Proposition \ref{AutomProd}\ref{AutProdRelPrimePoly} the automorphism group splits as a direct product of $Aut(\mathP)$ and $Aut(pt^{* k})$. 

Similar, if \P is not a prism, then in the unique prime factorisation of \P with respect to the Cartesian product the polytope $I$ does not occur. 
Hence, \P and $I$ are relatively prime and again by Proposition \ref{AutomProd}\ref{AutProdRelPrimePoly} the automorphism group splits as a direct product of $Aut(\mathP)$ and $Aut(I^{\times k})$. 
\end{proof}

If we know how the polytope splits and in addition if the polytope is not a prism or a pyramid, then we can decompose the automorphism group. 
We start by observing that a polytope of rank at least $2$ cannot be prism and pyramid at the same time.
For this we translate the property of being a pyramid in terms of abstract polytopes.

\begin{lem} \label{BeingPyramide}
Let \P be an abstract polytope. 
If \P is a pyramid, then there exists a vertex $v_0 \in \mathP$ such that for every vertex $v \in \mathP$ different from $v_0$, exists an edge $e \in \mathP$ incident with $v_0$ and $v$.
\end{lem}

\begin{proof}
Let \Q be a polytope with $\mathP = \mathQ * pt$. 
Then the vertices of \P are of the form $(0_{\mathQ}, 1_{pt})$ and $(w,0_{pt})$ for every vertex $w \in \mathQ$. 

Define $v_0 = (0_{\mathQ}, 1_{pt})$ and let $v \in \mathP$ be a vertex different from $v_0$. 
Hence, there exists a vertex $w \in \mathQ$ such that $v=(w,0_{pt})$. 
Then $e=(w,1_{pt})$ is an edge in \P and incident to $v_0$ and $v$.  
\end{proof}

We use this to prove that every prism is not a pyramid and vice versa. 
\begin{prop} \label{PrismNoPyrAndPyrNoPrism}
Let \P be an abstract polytope of rank at least 1. 
Then the prism $\mathP \times I$ is not a pyramid and the pyramid $\mathP * pt$ is not prism. 
\end{prop}

\begin{proof}
The polytope $I$ has exactly two vertices $x_1$ and $x_2$. 
Similarly as above, we conclude that the vertices of $P \times I$ are of the form $(v,x_i)$ for $i=1,2$ and every vertex $v \in \mathP$. 

Let $w$ be a vertex in $\mathP \times I$.
Then there exists a vertex $v_1$ in $\mathP$ such that $w=(v_1,x_j)$ where $j \in \{1,2\}$. 
Without loss of generality, assume that $j=1$ and hence $w=(v_1,x_1)$. 
By the diamond condition there exists another vertex $v_2 \in \mathP$ different from $v_1$. 
Let $(F,G)$ in $\mathP \times I$ be incident to $w$ and $(v_2,x_2)$, then $F$ is incident to $v_1$ and $v_2$ and hence has at least rank $1$. 
The same holds for $G$ which has also at least rank $1$ and thus $G=1_I$. 
By the definition of the rank function for the Cartesian product, it follows that $(F,G)$ has at least rank $2$. 
So $w$ and $(v_2,x_2)$ cannot be incident to an edge in $\mathP \times I$. 
Since $w$ is an arbitrary vertex, it follows by Lemma \ref{BeingPyramide} that $\mathP \times I$ is not a pyramid. 

Assuming that $\mathP * pt$ is a prism, there exists an abstract polytope $Q$ of the same rank as \P such that $\mathP * pt \cong \mathQ \times I$. 
But by the above $\mathQ \times I$ is not a pyramid and hence $\mathP * pt$ is not a prism. 
\end{proof}
It follows that every prism is prime with respect to the join and every pyramid is prime with respect to the Cartesian product.

\section{Automorphism groups of the inductively constructed polytopes}
Now we can describe the automorphism groups of the polytopes constructed by the inductive process pictured in Figure \ref{GenProcess}. 

Let \P be an abstract polytope constructed by the process in Figure \ref{GenProcess}. 
Then the last step was either a join or a Cartesian product.
\begin{enumerate}
\item[] First, assume that \P is a prism, i.e. the last step was a Cartesian product. 
\begin{enumerate}
\item If $\mathP = I^{\times k}$ for some $k \geq 1$, then
\begin{align*}
Aut(\mathP) \cong Aut(I^{\times k}) \cong (\mathbb{Z}/2\mathbb{Z})^k \rtimes Sym(k). 
\end{align*}
\item If there exists a pyramid \Q of rank at least $2$ in the construction process and $k \geq 1$ such that $\mathP = \mathQ \times I^{\times k}$, then \Q is by Proposition \ref{PrismNoPyrAndPyrNoPrism} not a prism and hence by Proposition \ref{AutNoPyrAndAutNoPrism}, we get 
\begin{align*}
Aut(\mathP) \cong Aut(\mathQ \times I^{\times k}) \cong Aut(\mathQ) \times \left( (\mathbb{Z}/2\mathbb{Z})^k \rtimes Sym(k)\right).
\end{align*}
\end{enumerate}
\item[] Second, assume that \P is a pyramid.
\begin{enumerate}
\item If $\mathP = pt^{* k}$ for some $k \geq 1$, then 
\begin{align*}
Aut(\mathP) \cong Aut(pt^{* k}) \cong Sym(k).
\end{align*} 
\item If there exists a prism \Q of rank at least 2 and $k \geq 1$ such that $\mathP = \mathQ * pt^{* k}$, then \Q is by Proposition \ref{PrismNoPyrAndPyrNoPrism} not a pyramid and hence with Proposition \ref{AutNoPyrAndAutNoPrism}, it follows that
\begin{align*}
Aut(\mathP) \cong Aut(\mathQ * pt^{* k}) \cong Aut(\mathQ) \times Sym(k). 
\end{align*}
\end{enumerate}
\end{enumerate}

Furthermore, we can describe a generating set for the subsequent polytopes of an abstract polytope \P of rank at least 2 in the construction process. 
Again we need to distinguish whether \P is a prism or a pyramid. 

Assume first that \P is a prism, i.e. the last step in the construction of \P was a Cartesian product. 
\begin{enumerate}
\item[($\times$)] Viewing the vertices of $I^{\times k}$ as vectors in $(\mathbb{Z}/2\mathbb{Z})^k$, the following permutations 
\begin{align*}
\sigma,s_i \colon (\mathbb{Z}/2\mathbb{Z})^k \to (\mathbb{Z}/2\mathbb{Z})^k
\end{align*}
for all $i=1, \dots, k-1$ form a generating set of $Aut(I^{\times k})$, where
\begin{align*}
\sigma( \left(z_1,z_2,  \dots, z_k\right) ) &= \left(z_1 +1, z_2, \dots, z_k\right) \\
s_i(\left(z_1, z_2, \dots, z_k\right)) &= \left(z_{\tau_i(1)},z_{\tau_i(2)},  \dots, z_{\tau_i(k)}\right)
\end{align*}
for every $i=1,\dots, k-1$, and where $\tau_i =(i~ i+1) \in Sym(k)$ is the transposition swapping $i$ and $i+1$. 

Hence, adding the permutation swapping the last two entries of the vectors to a generating set of $Aut(\mathP)$ gives a generating set of the automorphism group of $\mathP \times I$. 
\item[($*$)] Since \P is a prism, it is not a pyramid and thus $Aut(\mathP * pt) \cong Aut(\mathP)$. 

Hence, the generating set of $Aut(\mathP)$ is also a generating set of $Aut(\mathP * pt)$. 
\end{enumerate}
Assume now that \P is a pyramid.
\begin{enumerate}
\item[($\times$)] As \P is a pyramid, it is not a prism and thus $Aut(\mathP \times I) \cong Aut(\mathP) \times \mathbb{Z}/2\mathbb{Z}$. 

Hence, adding a permutation swapping one copy of \P with the other copy of \P to a generating set of $Aut(\mathP)$ gives a generating set of $Aut(\mathP \times I)$. 
\item[($*$)] The symmetric group $Sym(k)$ is generated by the transpositions $\tau_i = ( i~ i+1)$ for every $i=1, \dots, k-1$. 
Adding the transposition $\tau_k$ we get a generating set for $Sym(k+1)$. 

Hence, adding the permutation swapping the last two added copies of $pt$ to a generating set of $Aut(\mathP)$ leads to a generating set of $Aut(\mathP * pt)$. 
\end{enumerate}

\begin{ex}
We compute the automorphism group of $\mathP = \left( \left( I*pt\right) \times I^{\times 3} \right) * pt^{* 2}$. 
Tracing back we find $\mathQ_1 = \left( I * pt \right) \times I^{\times 3}$, which is a prism and hence not a pyramid.
So the unique factorisation of \P into prime polytopes with respect to the join is $\mathP = \mathQ_1 * pt^{* 2}$.
But we do not know the automorphism group of $\mathQ_1$ and thus we need to repeat the process.  
Again, tracing back leads to $\mathQ_2 = I*pt$ which is a pyramid and not a prism. 
Furthermore, the unique factorisation of $\mathQ_1$ into prime polytopes with respect to the Cartesian product is $\mathQ_1 = \mathQ_2 \times I^{\times 3}$. 
Moreover, since $I = pt * pt$, the automorphism group $Aut(\mathQ_2)$ is isomorphic to $Sym(3)$. 

By Proposition \ref{AutNoPyrAndAutNoPrism} we can compute the automorphism group of $\mathQ_1$.
\begin{align*}
Aut(\mathQ_1) \cong Aut(Q_2) \times Aut(I^{\times 3}) \cong Sym(3) \times \left( (\mathbb{Z}/2\mathbb{Z})^3 \rtimes Sym(3) \right)
\end{align*}
Using this, we can finally compute with Proposition \ref{AutNoPyrAndAutNoPrism} the automorphism group of $\mathP$.
\begin{align*}
Aut(\mathP) \cong Aut(Q_1) \times Aut(pt^{* 2}) \cong \left( Sym(3) \times \left( (\mathbb{Z}/2\mathbb{Z})^3 \rtimes Sym(3) \right) \right) \times Sym(2)
\end{align*} 
The automorphism groups of the subsequent polytopes of \P are then 
\begin{align*}
Aut(\mathP \times I) &\cong \left( \left( Sym(3) \times \left( (\mathbb{Z}/2\mathbb{Z})^3 \rtimes Sym(3) \right) \right) \times Sym(2) \right) \times \mathbb{Z}/2\mathbb{Z} \\
Aut(\mathP * pt) &\cong \left( Sym(3) \times \left( (\mathbb{Z}/2\mathbb{Z})^3 \rtimes Sym(3) \right) \right) \times Sym(3).
\end{align*} 
\end{ex}

\section{An algorithm to compute the automorphism groups}
We now give an algorithm attaching to each polytope three data points. 
Let $A(\mathP)$ be a variable to save the automorphism group of the last polytope constructed via the other product than $\mathP$, i.e. if \P is constructed via a Cartesian product (join), then the automorphism group of the last polytope constructed via a join (Cartesian product) is saved. 
Let $k(\mathP)$ be an integer and $prod(\mathP)$ a boolean to save which product was used to construct $\mathP$. 
The value $cartesian$ stands for the Cartesian product and $join$ for the join. 

At the beginning, we need to define the data attached to $I$ and $I*pt$. 
Start with $I$ and $I*pt$ and define 
\begin{align*}
&A(I)= null, k(I)=1, prod(I)=cartesian, \\
&A(I*pt)= null, k(I*pt)=3, prod(I*pt) = join. 
\end{align*}

Now we can attach the data to the remaining polytopes inductively.
For this let \P be an already constructed polytope and let $null \times A = A$. 
Then set the following variables for the subsequent polytopes of $\mathP$. 
\begin{enumerate}
\item[$\mathP \times I:$] If $prod(\mathP)=cartesian$, then
\begin{align*}
&A(\mathP \times I) = A(P) \\
&k(\mathP \times I)= k(\mathP) +1 \\ 
&prod(\mathP \times I) = cartesian,
\intertext{else}
&A(\mathP \times I) = A(P) \times Sym(k(\mathP)) \\
&k(\mathP \times I)= 1 \\ 
&prod(\mathP \times I) = cartesian. 
\end{align*}
\item[$\mathP * pt:$] If $prod(\mathP)=cartesian$, then
\begin{align*}
&A(\mathP *pt) = A(P) \times \left( (\mathbb{Z}/2\mathbb{Z})^{k(\mathP)} \rtimes Sym(k(\mathP)) \right) \\
&k(\mathP * pt)= 1 \\ 
&prod(\mathP * pt) = join,
\intertext{else}
&A(\mathP * pt) = A(P) \\
&k(\mathP * pt)= k(\mathP) + 1 \\
&prod(\mathP * pt) = join. 
\end{align*}
\end{enumerate}

With this data the automorphism group of a polytope \P constructed via this process can be computed as follows. 

If $prod(\mathP)=cartesian$, then
\begin{align*}
Aut(\mathP) &\cong A(P) \times \left( (\mathbb{Z}/2\mathbb{Z})^{k(P)} \rtimes Sym(k(\mathP)) \right)
\intertext{else, i.e. if $prod(\mathP)=join$, then } 
Aut(\mathP) &\cong A(P) \times Sym(k(\mathP)).
\end{align*}

\subsection*{Acknowledgments}
I would like to thank my advisor Linus Kramer for many helpful comments and remarks, and for the encouragement to write this article.


\end{document}